\documentclass[12pt]{amsart} 
\usepackage{graphicx}
\usepackage{bigints}
\usepackage{amsmath,amsthm}
\usepackage{amsfonts}
\usepackage{amssymb}
\usepackage[english]{babel}
\usepackage{xcolor}
\usepackage{comment}
\usepackage{enumitem}
\usepackage{calc}
\newlength{\depthofsumsign}
\makeatletter
\@namedef{subjclassname@2020}{%
	\textup{2020} Mathematics Subject Classification}
\makeatother

\newtheorem{theorem}{Theorem}[section]
\newtheorem{corollary}[theorem]{Corollary}
\newtheorem{lemma}[theorem]{Lemma}
\newtheorem{proposition}[theorem]{Proposition}

\theoremstyle{definition}
\newtheorem{definition}[theorem]{Definition}

\newtheorem{remark}[theorem]{Remark}
\newtheorem{question}[theorem]{Question}

\numberwithin{equation}{section}

\usepackage[a4paper,left=3cm,right=3cm,top=3cm,bottom=3cm]{geometry}

\newcommand{\norm}[1]{\left\Vert #1\right\Vert}

\newcommand{\DD}{{\mathbb D}}
\newcommand{\M}{{\mathcal M}}

\newcommand{\TT}{{\mathbb T}}

\newcommand{\C}{{\mathbb C}}

\newcommand{\N}{{\mathbb N}}
\newcommand{\Z}{{\mathbb Z}}
\newcommand{\R}{{\mathbb R}}
\newcommand{\T}{{\mathbb T}}

\newcommand{\OO}{O}

\newcommand{\nsum}[1][1.4]{
	\mathop{%
		\raisebox
		{-#1\depthofsumsign+1\depthofsumsign}
		{\scalebox
			{#1}
			{$\displaystyle\sum$}%
		}
	}
}
\newcommand{\bsum}[1][1.2]{
	\mathop{%
		\raisebox
		{-#1\depthofsumsign+1\depthofsumsign}
		{\scalebox
			{#1}
			{$\displaystyle\sum$}%
		}
	}
}

\begin{document}

	\title[Strongly Kreiss bounded operators on  $L^p$-spaces]
	{On the growth rate of powers of a strongly Kreiss bounded operator on  $L^p$-spaces}
	
	\author{Loris Arnold}
	\address{UMR CNRS 6139, Laboratoire de math\'ematiques Nicolas Oresme, Univ Normandie}
	\email{lfj.arld@gmail.com}

	\author{Christophe Cuny}
	\address{UMR CNRS 6205, Laboratoire de Math\'ematiques de Bretagne Atlantique, Univ Brest }
	\email{christophe.cuny@univ-brest.fr}


	\begin{abstract}
		Let $T$ be a strongly Kreiss bounded linear operator on $L^p$. We obtain a bound on  the rate of growth of the norms of the powers of $T$. The bound is optimal with respect to the polynomial scale. The proof makes use of Fourier multipliers, in particular of the Littlewood-Paley inequalities on arbitrary intervals as initiated by Rubio de Francia and developed by  Kislyakov and Parilov.
	\end{abstract}
	
	\maketitle
	\vspace*{-0.6 cm}

	\medskip
	\section{Introduction}
	
	Let $T$ be a bounded linear operator on a Banach space $X$. We study the 
	growth rate of $\|T^n\|$ when $X=L^p(\M,\Sigma,\mu)$, where $(\M,\Sigma, \mu)$ is a $\sigma$-finite measure space and $T$ 
    satisfies the so-called strong Kreiss
    condition. In the sequel, we shall simply write $L^p(\M)$ or even $L^p$ when no confusion is possible.
	
	\medskip
	
	To put our result into a proper context, let us first review several basic statements on the asymptotics of powers of Kreiss operators.
	
	\medskip
	
	In \cite{Kreiss}, Kreiss introduced the following conditions:
	\begin{equation}\label{Kreiss}
		\sigma(T) \subset  \overline{\mathbb{D}} \quad \text{and} \quad \|(\lambda-T)^{-1}\|\le \frac{C}{|\lambda|-1}\, , \quad |\lambda|>1\, ,
	\end{equation}
	where $\DD$ denotes the unit disk and $\sigma(T)$ stands for the spectrum of $T$. Moreover, Kreiss proved that if 
	$X$ is  a \emph{finite dimensional} Banach space, 
	\eqref{Kreiss}  is equivalent to the power-boundedness of $T$, i.e., $\sup_{n\in \N}\|T^n\|<\infty$.

	\medskip
	
	Later, McCarthy \cite{McCarthy} considered the following strengthening of \eqref{Kreiss}:
	\begin{equation}\label{iter-Kreiss}
		\sigma(T) \subset  \overline{\mathbb{D}}\quad \text{and} \quad  \|(\lambda -T)^{-k}\|\le \frac{C}{(|\lambda|-1)^k}\, , \quad |\lambda|>1, \, k\in \N\, ,
	\end{equation}
	known as the strong Kreiss condition (or iterated Kreiss condition). The operators satisfying \eqref{iter-Kreiss} will be called strongly Kreiss
	bounded. We denote by $C_{SK}$ the smallest constant $C>0$ such that \eqref{iter-Kreiss} holds.

	\medskip
	
	Lubich and Nevanlinna \cite{LN} proved that \eqref{Kreiss} implies $\|T^N\|={\underset{N\rightarrow \infty}{\mathcal{O}}}(N)$ and, by a result of Shields (see \cite[Theorem 1 and Proposition 3]{Shields}), this is optimal. 
	
	\medskip

	Moreover, Lubich and Nevanlinna also proved (see \cite[Theorem 2.1 and Example 2.2]{LN}) that \eqref{iter-Kreiss} 
	implies $\|T^N\|=\underset{N\rightarrow \infty}{\mathcal{O}}(\sqrt N)$, and that this estimate is the best possible for general Banach spaces. In \cite[Proposition 1.1]{Nevanlinna}, Nevanlinna proved that an operator $T$ satisfies the strong Kreiss condition  if and only if there exists $L>0$ such that 
	\begin{equation}\label{strong-kreiss}
		\|{\rm e}^{z T}\|\le L{\rm e}^{|z|}\, , \quad  z\in \C\,,
	\end{equation}
	and it turns out that the smallest constant $L>0$ such that \eqref{strong-kreiss} holds is $C_{SK}$ (see the sentence preceding \cite[formula (2)]{GZ}).
	When $X$ is a Hilbert space, 
	\eqref{strong-kreiss} implies $\|T^N\|=\underset{N\rightarrow \infty}{\mathcal{O}}((\log N)^\kappa)$, for some $\kappa>0$ (see \cite[Theorem 4.5.]{CCEL}). Moreover, by \cite[Proposition 4.9.]{CCEL}, for any $\kappa>0$ and any $1\le p < \infty$, the bounded operator $T_{\kappa}$ on $\ell^p(\N)$ defined by
	\begin{equation}\label{exampleStrKreiss}
		T_{\kappa}\big((x_n)_{n\in \N}\big) = \Big(\frac{\log^{\kappa}(n+2)}{\log^{\kappa}(n+1)}x_{n+1}\Big)_{n \in \N} 
	\end{equation}
	is a strongly Kreiss bounded operator $T$ such that for every $N\in \N$
	$$\|T_{\kappa}^N\|= \frac1{\log^{\kappa} (2)}\log^{\kappa} (N+2).
	$$

	\medskip
	
	Norm bounds for powers of operators 
	satisfying the Kreiss condition were also obtained in \cite{CCEL}. In \cite{Cuny}, these bounds were extended to $L^p$-spaces, $1< p<\infty$,
	(and more generally, to UMD Banach spaces with non-trivial type and/or cotype, see \cite{Cuny} for the definition and properties of UMD Banach spaces). However, the norm estimates for powers of strongly Kreiss bounded operators on $L^p$ spaces have not been addressed in \cite{Cuny}, and it is the purpose of this paper to study these estimates thoroughly. In particular, we obtain the next statement, which is the main result of the paper.
	
	\begin{theorem}\label{main-theorem}
		Let $T$ be a strongly Kreiss bounded operator on $L^p(\M)$, $1<p<\infty$, and let $\tau_p:=\left|\frac{1}{2}-\frac{1}{p}\right|$. There exist constants $C,\kappa>0$ depending only on $p$ and $C_{SK}$ such that for every
		$N\in \N$,
		\begin{equation}\label{main-estimate}
			\norm{T^N} \le C N^{\tau_p}\log^\kappa (N+1)\, .
		\end{equation}
	\end{theorem}
	
	\medskip
	
	\noindent {\bf Remarks.} When $p=2$ we recover the optimal result from \cite{CCEL}. The theorem does not cover the cases $p\in \{1,\infty\}$ but in these cases, the bound of Lubich and Nevanlinna is the best possible, as we show in Proposition \ref{example} below. Finally, let us notice that the case of  UMD Banach spaces has been investigated very recently by Deng, Lorist and Veraar 
	\cite{DLV}. They prove (see  \cite[Corollary 3.2]{DLV}) that any strongly Kreiss bounded operator $T$ on a UMD Banach space satisfies $\|T^N\|= \underset{N\to \infty}\OO(N^\alpha)$ for some $\alpha<1/2$.
	
	\medskip

	
	
	It follows from Proposition \ref{example} below that the exponent $\tau_p$ in \eqref{main-estimate} is the best possible. 
	The proposition is a generalization of an example of Lubich and Nevanlinna \cite{LN} who considered the case 
	$p=\infty$. 

\medskip	
	
	Here and in the
	sequel, let $B(X)$ denote the space of bounded operators on $X$.
	
	\begin{proposition}\label{example}
		There exists an operator $V$ strongly Kreiss bounded on every $\ell^p(\Z)$, $1\le p\le\infty$ such that 
		\begin{equation}\label{upper-lower}C^{-1} N^{|1/2-1/p|} \le \|V^N\|_{B(\ell^p(\Z))}\le C N^{|1/2-1/p|}\, ,
		\end{equation}
		for a constant $C\ge 1$ and all $N \in \N$.
	\end{proposition}

	
	
	While Proposition  \ref{example} proves that the exponent
	$\tau_p$ in \eqref{main-estimate} is sharp, we do 
	not know whether the extra logarithmic terms are needed in general. The example \eqref{exampleStrKreiss} proves that one cannot avoid the logarithmic terms when $p=2$. Hence it is natural to conjecture that this is also the case for $1<p<\infty$. 
	
	\medskip
	
	Another question of interest is whether one can find an optimal bound for
	$\kappa$ (and the implicit constant in the notation $\OO$). For $p=2$, it is possible to provide an explicit bound on these constants, following carefully the arguments in \cite[p. 13-15]{CCEL}, but we do not know whether the obtained $\kappa$ will be optimal. For $p\neq 2$, 
	the constant $\kappa$ depends on several Fourier multiplier norms in $L^p$ and it is
	hard to estimate it in an optimal way. On the other hand, some of these constants
	tend to $\infty$ as $p$ is getting close  to 1 (or to $\infty$). For instance, the best possible constant $R_p$ in \eqref{riesz}  tends to $\infty$ when $p$ is getting close  to 1 (or to $\infty$).
	
	\medskip

	Let us mention that, when $T$ is a \emph{positive} strongly Kreiss bounded operator on $L^p(\M)$ and $p\in [1,4/3)\cup (4,\infty]$,  it is possible to improve \eqref{main-estimate} by replacing 
	$N^{|1/2-1/p|}$ with $N^{1/\bar p}$, $\bar p=\max (p,p/(p-1))$, see Proposition \ref{positive} below.
	
	\medskip

	Finally, we describe the organization of the paper. In Section 2 we state and prove auxiliary results concerning Fourier multipliers in $L^p$. Section 3 is dedicated to the proof of Theorem \ref{main-theorem} and in Section 4 we prove Proposition 
	\ref{example}. In Section 5 we provide improved estimates on the growth rate of $\|T^n\|$ for strongly Kreiss bounded operators under extra assumptions on $T$. In 
	particular, we consider the case where $T$ is a positive operator on $L^p(\M)$, $1\le p\le \infty$, hence including the cases 
	$p\in \{1,\infty\}$.

	\medskip

	\section{Auxiliary results}
	Through the paper we will denote by $\T=[-\pi,\pi)$ the torus and  by $\lambda$ the Haar-Lebesgue measure on $\T$. We also set $e_n(t):={\rm e}^{i nt}$ for all $t\in \T$ and $n \in \Z$.

	\medskip

	Given a bounded interval $I\subset \Z$,  we define an operator $M_I$ on $L^p(\TT)$, by setting for every $f\in L^p(\TT)$,
	$$
	M_If(t):= \sum_{i\in I}c_i(f) e_i(t) \, , \quad t\in \TT
	$$
	where $\sum_{n\in \Z} c_n(f)e_n$ is the formal Fourier series of $f$.
	
	\medskip
	
	Recall the definition of the weak $L^1$-norm $\norm{\cdot}_{L^{1,\infty}(\T)}$ on $(\T,\lambda)$. For every measurable function $g$ on $\T$, we set
	$$
	\|g\|_{L^{1,\infty}(\T)}:=\sup_{x>0}x\lambda(|g|\ge x)= 
	\sup_{x>0}x \lambda(\{t \in \T\, :\, |g(t)|\ge x\})\, .
	$$

	\begin{proposition}\label{new-prop}
		Let $1<p<\infty$ and $p'=\min(2,p)$.  There exists $D_p>0$  such that for every finite collection 
		$(I_\ell)_{1\le \ell\le L}$ of mutually disjoint intervals of integers, 
		\begin{equation}\label{inequality1}
			\Big\|\Big( \sum_{\ell=1}^L |M_{I_\ell}f|^2\Big)^{1/2}\Big\|_{L^p(\TT)} 
			\le D_p L^{1/p'-1/2}\|f\|_{L^p(\TT)}\, , \quad  f\in L^p(\TT)\,.
		\end{equation}
		Furthermore,  there exists $D_{1}>0$  such that for every finite collection 
		$(I_\ell)_{1\le \ell\le L}$ of mutually disjoint intervals of integers, 
		\begin{equation}\label{inequality-L1-1}
			\Big\|\Big( \sum_{\ell=1}^L |M_{I_\ell}f|^2\Big)^{1/2}\Big\|_{L^{1,\infty}(\T)} 
			\le D_{1} L^{1/2}\|f\|_{L^1(\TT)}\, , \quad f\in L^1(\TT)\, .
		\end{equation}	
	\end{proposition}
	

	\begin{proof} For $p\ge 2$, \eqref{inequality1} is the so-called Rubio de  Francia inequality, see 
		\cite{Rubio}. Actually, Rubio de Francia proved the result on the real line. The result for the torus (with extensions) appears in Kislyakov-Parilov \cite{KP}.

		Inequality \eqref{inequality-L1-1} appears in \cite{KP} (take $g_k=f$ in the first inequality \cite[Section 1.4 p.6419]{KP}), see also \cite[Exercise 6.4.1 (a), p.337]{Grafakos2} for a version on the real line.

		Then \eqref{inequality1} for $1<p<2$ follows by applying the  Marcinkiewicz interpolation theorem, see \cite[Theorem 1.3.2]{Grafakos2}, to the sub-linear operator $f\mapsto \Big( \sum_{\ell=1}^L |M_{I_\ell}f|^2\Big)^{1/2}$ with $p_0=1$ and $p_1=2$. 
	\end{proof}
	
	\begin{corollary}\label{corollary}
		Let $1< p<\infty$. Let $p''=\max(2,p)$ and recall that $p' = \min(2,p)$. For every finite collection 
		$(I_\ell)_{1\le \ell\le L}$ of disjoint and consecutive  intervals of integers  and every $f\in L^p(\TT)$, setting $I:=\cup_{\ell =1}^LI_\ell $, we have
				\begin{align}\label{inequality3}
			\|M_I f\|_{L^p(\TT)} &  
			\le D_q  L^{1/2-1/p''}\Big\|\Big( \sum_{\ell=1}^L |M_{I_\ell}f|^2\Big)^{1/2}\Big\|_{L^p(\TT)} \nonumber \\
			& \le D_q  L^{1/2-1/p''}\Big( 
			\sum_{\ell=1}^L \|M_{I_\ell}f\|_{L^p(\TT)}^{p'}\Big)^{1/p'},
		\end{align} 
		where $q = p/(p-1)$ and $D_{q}$ is defined in Proposition \ref{new-prop}.
	\end{corollary}
	
	\noindent {\bf Proof.} Clearly, it is enough to assume that $f$ is a trigonometric polynomial supported in $I$. 
	Let $g \in L^q(\TT)$, where $q=p/(p-1)$. Note that $q'=\min(p/(p-1),2)=p''/(p''-1)$ and that $1/q'-1/2=1/2-1/p''$. Using orthogonality and the
	Cauchy-Schwarz and H\"older inequalities, we see that
	\begin{align*}
		\Big|\int_\T f\bar g \, d\lambda\Big| &=\Big|\int_\T\sum_{\ell=1}^L M_{I_\ell}f M_{-I_\ell}\bar g\, d\lambda\Big|\\
		&\le \Big\|\Big(\sum_{\ell=1}^{L}|M_{I_\ell}f|^2\Big)^{1/2}\Big\|_{L^p(\TT)}
		\Big\|\Big(\sum_{\ell=1}^{L}|M_{-I_\ell}\bar g|^2 \Big)^{1/2}\Big\|_{L^q(\TT)}\\
		&\le D_{q}  L^{1/2-1/p''}\Big\|\Big(\sum_{\ell=1}^{L}|M_{I_\ell}f|^2\Big)^{1/2}\Big\|_{L^p(\TT)}\|\bar g\|_{L^q(\TT)} \, ,
	\end{align*}
	where we used Proposition \ref{new-prop}. 
	Then \eqref{inequality3} follows by taking the supremum over $g\in L^q(\TT)$, with $\|g\|_{L^q(\TT)}=1$. 
	The last estimate follows by using that $x\mapsto x^{p/2}$ is subadditive on $\R_+$ when $p\le 2$ and Minkowski's inequality in 
	$L^{p/2}(\TT)$ when $p> 2$. \hfill $\square$
	
	\medskip

	Corollary \ref{corollary} implies the following statement. For the remainder of this paper, we will use, without explicit mention, the isomorphism between $L^p(\M\times{\mathcal N})$ and $L^p(\M,L^p({\mathcal N}))$, for measure spaces $(\M,\mu)$ and $({\mathcal N},\nu)$.


	\begin{corollary}\label{lemma}
		Let $1 < p < \infty$. There exists $C_p >0$ such that for every $N\in \N $ and each $(x_1, \ldots, x_{N^2})\in 
		(L^p(\M))^{N^2}$,	
		\begin{align}\label{lemsum}
			&\Big\| \sum_{k=1}^{N^2}  e_kx_k\Big\|_{L^p(\T,L^p(\M))}\, 
			\\  & \quad \leq C_p  N^{1/2-1/p''}\left(\sum_{n=0}^{N-1}\norm{\sum_{k=n^2+1}^{(n+1)^2}e_kx_k}^{p'}_{L^p(\T,L^p(\M)) }\right)^{1/p'}.\nonumber
		\end{align}
		
	\end{corollary}
	
	\begin{proof}
		Let $N \in \N$ and $I_n := [n^2+1, (n+1)^2)$ for $n \in \{0,\cdots, N-1\}$. By Corollary \ref{corollary}, for all $(a_1, \ldots, a_{N^2})\in \C^{N^2}$ and $f \in L^p(\TT)$, 
		\begin{equation}\label{ineq}	\norm{\sum_{k=1}^{N^2} a_ke_k}_{L^p(\TT)}   \le D_q N^{1/2-1/p''}\left( 
			\sum_{n=0}^{N-1} \norm{\sum_{k=n^2+1}^{(n+1)^2} a_ke_k}_{L^p(\TT)}^{p'}\right)^{1/p'}\, .
		\end{equation}
		
		Assume that $p>2$ and apply \eqref{ineq} with $a_k=x_k(\omega)$, $\omega\in \M$. Then, using Fubini's theorem, 
		we obtain that
		\begin{gather*}
			\norm{ \sum_{k=1}^{N^2}  e_kx_k
			}_{L^p(\T,L^p(\M))}^p =\int_\M \Big( \int_\T\Big|
			\sum_{k=1}^{N^2} x_k(\omega) e_k \Big|^p \, 
			d\lambda\, \Big)d\mu(\omega)\\
			\le \big(D_q N^{1/2-1/p''}\big)^p \bigintsss_\M \Bigg( \,
			\sum_{n=0}^{N-1} \norm{\sum_{k=n^2+1}^{(n+1)^2} 
				x_k(\omega)e_k }_{L^p(\TT)}^{2}\, \Bigg)^{p/2}
			\, d\mu(\omega)\, .
		\end{gather*}
		We conclude by using Minkowski's inequality in $L^{p/2}(\M)$ and Fubini's theorem again. 
		
		The case where $1\le p \le 2$, may be handled similarly, using Fubini's theorem only once.  
		
		Hence we get the result with $C_p = D_q$.
		
	\end{proof}

	We conclude this section by an auxiliary statement which will be used frequently throughout the paper. In order to interpret condition \eqref{technical2}, we first recall the notion of a sequence of bounded variation. 
	
	\begin{definition}
		We say that a sequence of complex numbers ${\bf a}=(a_n)_{n\in \Z}$ is of bounded variation if $v({\bf a}):= \sum_{n\in \Z} |a_{n+1}-a_{n}|<\infty$. Then, we call $v({\bf a})$, the variation of ${\bf a}$.
	\end{definition}
	In the remainder of the paper, $\sum\limits_{\alpha\le n\le \beta}$ for $\alpha,\beta \in \R$ will denote the sum over integers $n \in \Z$ such that $\alpha\le n \le \beta$.
	\begin{lemma}\label{technical}
		There exists $C_{exp}>1$ such that for every $N\in \N$ and every 
		integer $ K\in [N-6\sqrt{N}, N+1]$, 
		\begin{equation}
			\label{technical1}\frac{{\rm e}^N}{C_{exp} \sqrt N}\le \frac{N^{K}}{K!}
			\le \frac{C_{exp}{\rm e}^N}{ \sqrt N}\, ,
		\end{equation}
		and	
		\begin{equation}	\label{technical2}
			\nsum_{N-2\sqrt{N}+2\le n\le N }\Bigg|
			\Big(\bsum_{n-4\sqrt N\le k\le n} \frac{N^{k}}{k!}\Big)^{-1}-
			\Big(\bsum_{n+1-4\sqrt N\le k\le n+1}\frac{N^k}{k!}\Big)^{-1}\Bigg|\le \frac{C_{exp}}{{\rm e}^{N}}\, . 
		\end{equation}
		In particular, the sequences 
		$$
		\Bigg\{{\rm e}^N\Big(\bsum_{n-4\sqrt N\le k\le n} \frac{N^{k}}{k!}\Big)^{-1},\; N-2\sqrt{N}+2 \le n\le N  \Bigg\}
		$$ are of
		 uniformly bounded variation.
	\end{lemma}

	\noindent {\bf Proof.} The upper bound of \eqref{technical1} 
	follows from  \cite[Lemma 3.4]{CCEL} and the lower bound may be proved similarly. Then, \eqref{technical2} follows from the fact that, for $ N-2\sqrt{N}+2 \le n\le N$, writing $m:=\lfloor n-4\sqrt N\rfloor$+1,  where $\lfloor\cdot\rfloor$ denotes the integer part,  we have
	\begin{align*}
		\Bigg|\Big(\bsum_{n-4\sqrt N\le k \le n}\frac{N^{k}}{k!}\Big)^{-1}-\Big(\bsum_{n+1-4\sqrt N\le k \le n+1}\frac{N^{k}}{k!}\Big)^{-1}\Bigg| & \le \displaystyle\frac{\frac{N^{m}}{
				m!}+\frac{N^{n+1}}{(n+1)!}}{\Big(\displaystyle\bsum_{n+1-4\sqrt N\le k \le n}\frac{N^{k}}{k!}\Big)^2} \\
		&  \le \frac{\tilde C e^{-N}}{\sqrt N}\, ,
	\end{align*}
	for a constant $\tilde C>0$, where we used \eqref{technical1}.
	\hfill $\square$

	\medskip

	\section{Proof of Theorem \ref{main-theorem}}
	The proof of Theorem \ref{main-theorem} makes use of Fourier multipliers, 
	see  \cite[p. 11-12]{Cuny} for a brief description of Fourier multipliers in UMD Banach spaces (in particular in $L^p$ spaces). Actually, we only make use of
	\emph{real-valued} Fourier multipliers in our proofs and then use Fubini's theorem to obtain results for $L^p$-valued 
	Fourier multipliers.
	
	\medskip
	
	We shall rely on the following well-known results on $L^p$-multipliers used in a similar context in \cite{Cuny}.
	\begin{lemma}\label{RieszStechkin}
		Let $X$ be a UMD space and $1<p<\infty$. 
		\begin{enumerate}[label=(\roman*)]
			\item \textit{Riesz's theorem} :  there exists $R_p>0$ such that for every interval $I \subset \Z$ and every finitely supported sequence $(c_i)_{i\in \Z}\subset  X$ we have
			\begin{equation}\label{riesz}
				\norm{\sum_{i\in I}e_ic_i}_{L^p(\T,X)} \le R_p\norm{\sum_{i\in \Z}e_ic_i}_{L^p(\T,X)}. 
			\end{equation} 
			\item \textit{Stechkin's theorem} : there exists $S_p>0$ such that for every sequence of complex numbers $(a_i)_{i\in \Z}$   with bounded variation and every finitely supported sequence $(c_i)_{i\in \Z}\subset X$,
			$$
			\norm{\sum_{i\in \Z}a_ie_ic_i}_{L^p(\T,X)} \le S_p\Big(\|(a_i)_{i\in \Z}\|_{\ell^\infty(\Z)}+v((a_i)_{i\in \Z})\Big)\norm{\sum_{i\in \Z}e_ic_i}_{L^p(\T,X)}. 
			$$
		\end{enumerate}
	\end{lemma}

To prove our main result, we need two key lemmas. The first one shows that if one controls the $L^p(\TT,L^p(\M))$-norms of sums of the form \\ $\sum\limits_{N + 2 - 2\sqrt{N} \leq n \leq N} e_n T^n x$ with a growth rate of $N^{\alpha}$, $\alpha >0$, then it is possible to deduce a bound for the $L^p(\TT,L^p(\M))$-norms of $\sum\limits_{1 \leq n \leq N} e_n T^n x$ with a slightly larger growth rate of $N^{\alpha + \delta_p}$. The second one, conversely, provides a bound on the $L^p(\TT,L^p(\M))$-norms of sums  $\sum\limits_{N + 2 - 2\sqrt{N} \leq n \leq N} e_n T^n x$ with a growth rate of $N^{\beta/2}$, assuming that the $L^p(\TT,L^p(\M))$-norms of $\sum\limits_{1 \leq n \leq N} e_n T^n x$ are controlled with a growth rate of $N^\beta$, $\beta >0$, and that $T$ is strongly Kreiss bounded on $L^p(\M)$.

	We now state and prove the first lemma. Recall that $p' = \min(p,2)$ and $p''=\max(p,2)$.
	
	\begin{lemma}\label{lemind1}
		Let $T$ be a bounded operator on $L^p(\M)$, $1< p<\infty$.
		Assume that there exist $D>0$ and $ 0<\alpha\le 1$,  such that for every $N\in \N$ and every $x\in L^p(\M)$, 
		\begin{equation}\label{assumtionlemind1}
			\norm{\sum_{N + 2 - 2\sqrt{N}\le n\le N} e_n T^n x}_{L^p(\T,L^p(\M))} \le D N^\alpha  \norm{x}_{L^p(\M)}\, .
		\end{equation}
		Then, setting $\delta_p:= \frac12\left(\frac2{p'}-\frac1p\right)$, there exists $E_p^{(1)}>0$ such that for every $N\in \N$ and every $x\in L^p(\M)$, we have
		$$
		\norm{\sum_{1\le n \le N}e_n T^n x}_{L^p(\T,L^p(\M))}\le DE_p^{(1)} N^{\alpha +\delta_{p}}\norm{x}_{L^p(\M)}
		\, .
		$$
		
	\end{lemma}
	
	\noindent {\bf Proof.}  By Corollary \ref{lemma} and the assumption, for every $M\in \N$ and every $x\in L^p(\M)$, 
	\begin{align*}
		\Bigg\|&\sum_{1\le n \le M^2}e_n T^n x\Bigg\|_{L^p(\T,L^p(\M))}^{p'} \\
		&\le\big(C_p  M^{1/2-1/p''}\big)^{p'}\left( \nsum_{n=0}^{M-1}\norm{\sum_{k=n^2+1}^{(n+1)^2}e_nT^nx}_{L^p(\T,L^p(\M))}^{p'} \right)  \\
		& \le (C_pD)^{p'}M^{p'(1/2-1/p'')}\sum_{n=0}^{M-1}(n+1)^{2p'\alpha}\norm{x}_{L^p(\M)}\\
		&  \le (C_pD)^{p'} M^{p'(1/2-1/p'')+2p'\alpha+1}\norm{x}_{L^p(\M)}\,.
	\end{align*}
	Let $N\in \N$ and $M:= \lfloor\sqrt N\rfloor+1$. By the Riesz theorem (Lemma \ref{RieszStechkin} $(i)$), using that 
	$N\le M^2\le 4N$, we infer that
	\begin{align*}
		\norm{\sum_{1\le n \le N}e_n T^n x}_{L^p(\T,L^p(\M))}&\le  R_p\norm{\sum_{1\le n \le M^2}e_n T^n x}_{L^p(\T,L^p(\M))} \\
		&\le  R_p C_pD(4N)^{(1/2 -1/p'')/2 + \alpha+1/(2p')}\norm{x}_{L^p(\M)} \, ,
	\end{align*}
	and the result follows since
	$$
	(1/2 -1/p'')/2 + \alpha+1/(2p') = \alpha + (1/2+1/p'-1/p'')/2 = \alpha + \delta_p.
	$$
	
	\hfill $\square$

	\medskip
We now state and prove our second key lemma in a general setting where we assume that the $L^p(\TT,L^p(\M))$-norms of  $\sum_{ n=1}^ N e_n T^n x$ are bounded by a growth rate of $f(N)$ where  $f : \R_+ \rightarrow (1,\infty)$ is a given function.

	\begin{lemma}\label{lemind2}
		Let $T$ be a strongly Kreiss bounded operator on $L^p(\M)$, $1<p <\infty$. Assume that there exist a function $f : \R_+ \rightarrow (1,\infty)$ and $D>0$ such that for every $N\in \N$ and every $x\in L^p(\M)$, 
		\begin{equation}\label{assumption}
			\norm{\sum_{ 1 \le n \le N } e_n T^n x }_{L^p(\T,L^p(\M))} \le D f(N)  \|x\|_{L^p(\M)}\, .
		\end{equation}
		Then, there exists $E_p^{(2)}$ such that for every $N\in \N$ and every $x\in L^p(\M)$,
		$$
		\norm{\sum_{N-2\sqrt{N}+2\le n \le N}e_n T^n x}_{L^p(\T,L^p(\M))}\le DE_p^{(2)} f(4\sqrt{N}) \|x\|_{L^p(\M)} \, .
		$$
	\end{lemma}
	
	\begin{proof} The proof is similar to the one of \cite[Lemma 4.7]{CCEL}.
		
		Let $x \in L^{p}(\M)$.	Firstly, since $T$ is strongly Kreiss bounded, using assumption \eqref{assumption}, we have, for all $N\in \N$,
		
		\begin{align}\label{exp1}
			\norm{ {\rm e}^{e_1 N T}\sum_{n = 1}^{4\sqrt{N}} e_nT^nx}_{L^p(\T,L^p(\M))}  &\le C_{SK}{\rm e}^N\norm{ \sum_{n = 1}^{4\sqrt{N}} e_nT^nx}_{L^p(\T,L^p(\M))}   \\
			&\le C_{SK}D{\rm e}^Nf(4\sqrt{N})\|x\|_{L^p(\M)} \nonumber.
		\end{align}
        Recall that $C_{SK}$ denotes the smallest constant $L>0$ such that \eqref{strong-kreiss} holds.
		
		Furthermore, for all $N\in \N$,
		
		\begin{equation*}
			e^{e_1 NT}\sum_{n = 1}^{4\sqrt{N}} e_nT^nx = \sum_{1\le n \le 4\sqrt{N}}e_n{T}^nx\sum_{0\le k \le n}\frac{N^k}{k!} + \sum_{ n \ge 4\sqrt{N}}e_n{T}^nx\sum_{n-4\sqrt{N}\le k \le n}\frac{N^k}{k!}.
		\end{equation*}
		Using the Riesz theorem (Lemma \ref{RieszStechkin} $(i)$), we have 
		\begin{align}\label{exp1bis}
			&R_p\norm{ e^{e_1 N T}\sum_{n = 1}^{\sqrt{N}} e_nT^nx}_{L^p(\T,L^p(\M))} \\
			& \qquad \ge  
			\norm{ \sum_{ N-2\sqrt{N}+2 \le n \le N }e_n{T}^nx\sum_{n-4\sqrt{N}\le k \le n}\frac{N^k}{k!}}_{L^p(\T,L^p(\M))} .\nonumber
		\end{align}
		Let \[
		a_n^N := \left( \sum_{n-4\sqrt{N} \leq k \leq n} N^k / (k!) \right)^{-1}, \quad N-2\sqrt{N}+2  \leq n \leq N,
		\]
		and define $\mathbf{a} := (a_n)_{n\in \Z}$ by 
		\[
		a_n := \begin{cases}
			a_n^N, &  \text{if } N-2\sqrt{N}+2 \leq n \leq N , \\
			0, &  \text{otherwise}.
		\end{cases}
		\]  By Lemma \ref{technical},
		$$
		\norm{\mathbf{a}}_{\ell^{\infty}(\Z)} \le  \Big(\sum_{N-4\sqrt{N}\le k \le N-2\sqrt{N}+2} \frac{e^N}{C_{exp}\sqrt{N}}\Big)^{-1} \le \frac{C_{exp}}{e^N}
		$$
		and
		$$
		v(\mathbf{a}) \le \frac{C_{exp}}{e^N}.
		$$
		Then, by Stechkin's theorem (Lemma \ref{RieszStechkin} $(ii)$), 
		\begin{align}\label{exp2}
			&2C_{exp}S_p \norm{ \sum_{ N -2 \sqrt{N}+2 \le n \le N }e_n{T}^nx\sum_{n-\sqrt{N}\le k \le n}\frac{N^k}{k!}}_{L^p(\T,L^p(\M))}  \\
			&\qquad \ge e^{N} \norm{ \sum_{ N -2 \sqrt{N}+2 \le n \le N }e_n{T}^nx}_{L^p(\T,L^p(\M))}, \nonumber
		\end{align}
		with $C_{exp}$ defined as in Lemma \ref{technical}. Combining \eqref{exp1}, \eqref{exp1bis} and \eqref{exp2}, we get the desired result.
	\end{proof}

	\medskip
    	
	From \cite[Theorem]{GZ} and the proof of  \cite[Theorem 3.1]{MSZ} we obtain the following lemma.
	\begin{lemma}\label{lemind3}
		Let $T$ be a strongly Kreiss bounded operator on a Banach space $X$. Then there exists $D_{SK}>0$ depending only on $C_{SK}$, such that
		$$
		\norm{\sum_{1\le n\le N} \lambda^n T^n}_X \le D_{SK} N, \quad \lambda \in \C, \quad|\lambda|= 1.
		$$
		
	\end{lemma}

	\medskip
    This statement enables us to establish the following estimate for a strongly Kreiss bounded operator on $L^p(\M)$: for every $N\in \N$ and every $x\in L^p(\M)$
    \begin{equation}\label{estisumTnLp}    
    \norm{\sum_{1\le n \le N} e_nT^n}_{L^p(\T,L^p(\M))} \le  D_{SK}N
    \end{equation}
    where $D_{SK}$ is given in the preceding lemma.  Subsequently, Lemmas \ref{lemind1} and \ref{lemind2} can be applied inductively to yield an improvement of the estimate \eqref{estisumTnLp}. This improvement is detailed in the next result.

	\medskip

	\begin{corollary}\label{final-corollary}
		Let $T$ be a strongly Kreiss bounded operator on $L^p(\M)$, $1<p <\infty$. Then, there exists $E_p>0$ such that for all integers $N\ge 1$, $K \ge 0$ and all $x\in L^p(\M)$, 
		\begin{equation}\label{after-induction}
			\norm{\sum_{1\le n\le N} e_n T^nx}_{L^p(\T,L^p(\M))} \le  D_{SK}E_p^{K+1}N^{2\delta_p +(1/2-\delta_p)2^{-K}} \|x\|_{L^p(\M)} \, .
		\end{equation}
		
	\end{corollary}
	
	\begin{proof}
		We proceed by induction on $K$. By \eqref{estisumTnLp}, $T$ satisfies \eqref{assumption} with $D = D_{SK}$ and $f(N) = N$, hence using Lemma \ref{lemind2} we infer that for all $N\in \N$ and $x\in L^p(\M)$,  
		\begin{align*}
			\norm{\sum_{ N-2\sqrt{N}+2\le n \le N}e_n T^n x}_{L^p(\T,L^p(\M))} &\le  D_{SK}E_p^{(2)} f(4\sqrt{N})\|x\|_{L^p(\M)} \\
			&= 4 D_{SK}E^{(2)}_pN^{1/2}  \|x\|_{L^p(\M)} \, .
		\end{align*}
		It follows that $T$ satisfies \eqref{assumtionlemind1} with $D = D_{SK}E^{(2)}_p$ and $\alpha = 1/2$. By Lemma \ref{lemind1}, for every $N\in \N$ and every $x\in L^p(\M)$,  
		\begin{align*}
			\norm{\sum_{1\le n\le N} e_n T^nx}_{L^p(\M)} &\le 4D_{SK}E_p^{(1)}E_p^{(2)}N^{1/2+\delta_p}\|x\|_{L^p(\M)} \\
			&= D_{SK}E_pN^{1/2+\delta_p}\|x\|_{L^p(\M)} 
		\end{align*}
		where $E_p = 4E_p^{(1)}E_p^{(2)}$. Thus we obtain \eqref{after-induction} with $K=0$.
		The induction step from $K-1$ to $K$ can be justified
		similarly. Let us assume that $T$ satisfies
		$$
		\norm{\sum_{1\le n\le N} e_n T^nx}_{L^p(\T,L^p(\M))} \le  D_{SK}E_p^{K}N^{2\delta_p +(1/2-\delta_p)2^{-(K-1)}} \|x\|_{L^p(\M)} \, .
		$$
		Then $T$ satisfies \eqref{assumption} with
		$$D = D_{SK}E_p^K\quad \text{and} \quad f(N) = N^{2\delta_p+(1/2 -\delta_p)2^{-(K-1)}}, $$ and Lemma \ref{lemind2} gives, for every $N\in \N$ and every $x\in L^p(\M)$,
		$$
		\norm{\sum_{1\le n\le N} e_n T^nx}_{L^p(\M)} \le 4D_{SK}E_p^{(2)}E_p^KN^{\delta_p+(1/2 -\delta_p)2^{-K} }\|x\|_{L^p(\M)}.
		$$
		
		Now, $T$ satisfies \eqref{assumtionlemind1} with $$D = 4D_{SK}E_p^{(2)}E_p^K\quad \text{and} \quad\alpha = \delta_p+(1/2 -\delta_p)2^{-K}. $$ 
		Lemma \ref{lemind1} then yields, for every $N\in \N$ and every $x\in L^p(\M)$,
		\begin{align*}
			\norm{\sum_{1\le n\le N} e_n T^nx}_{L^p(\M)} & \le 4D_{SK}E_p^{(1)}E_p^{(2)}E_p^KN^{\delta_p+(1/2 -\delta_p)2^{-K} + \delta_p}\|x\|_{L^p(\M)}  
			\\& =  D_{SK} E_p^{K+1}N^{2\delta_p +(1/2-\delta_p)2^{-K}} \|x\|_{L^p(\M)}. 
		\end{align*}
	\end{proof}
	We are now ready to prove  Theorem  \ref{main-theorem}.
	
	\begin{proof}[Proof of Theorem \ref{main-theorem}]

		Without loss of generality, we may and do assume that $N\ge 3$. Let $K\ge 0$ be the  integer such that $ 2^{K+1} \le\log(N+1) \le 2^{K+2}$. Then we have $K+1 \le \log (\log (N+1))/\log 2$ and 
		\begin{equation}\label{control-constant}
			(E_p)^{K+1} \le {\rm exp}(\log E_p \log (\log (N+1))/\log 2) = \log^{\kappa_p}(N+1) \, .
		\end{equation} 
        with 
        		$$
		\kappa_p := \frac{\log E_p}{\log 2}. 
		$$
		Moreover, since $2^{-(K+1)} \le 2/\log(N+1),$
		\begin{equation}\label{inter-est}N^{(1/2-\delta_p)2^{-K}}\le N^{2^{-(K+1)}} \le {\rm e}^{2^{-(K+1)}\log(N+1)} \le {\rm e}^2\, .
		\end{equation}
		
		Using \eqref{after-induction}, \eqref{control-constant} and \eqref{inter-est}, we infer that for every $N\in \N$ and every $x\in L^p(\M)$,
		\begin{equation}\label{intermediary-estimate}
			\norm{\sum_{1\le n\le N} e_n T^nx}_{L^p(\T,L^p(\M))}  
			\le e^2D_{SK} N^{2\delta_p}\log^{\kappa_p} (N+1) \|x\|_{L^p(\M)}\, .
		\end{equation}

		Using that $T^*$ is strongly Kreiss bounded on $L^q(\M)$, $q=p/(p-1)$, we obtain a similar estimate for $T^*$, that is, for every $N\in \N$ and every $x\in L^q(\M)$,
		\begin{equation}\label{intermediary-qestimate}
			\norm{\sum_{1\le n\le N} e_n {T^*}^nx^*}_{L^q(\T,L^q(\M))}  
			\le e^2D_{SK} N^{2\delta_q}\log^{\kappa_q} (N+1)\|x^*\|_{L^q(\M)}\, .
		\end{equation}

		It follows that for all $x\in L^p(\M)$ and $x^*\in L^q(\M)$,
		\begin{align*}
			N \big|\langle x^*, T^{N}x\rangle\big| \\
			&= \Big|\int_\T \Big\langle \sum_{1\le n\le  N}
			e_n T^{*n}x^*, \sum_{1\le m\le N}\bar e_m 
			T^{N+1-m} x\Big\rangle \, d\lambda(\gamma) \Big| \\
			&\le \norm{\sum_{1\le n\le N} e_n T^{*n} x^*}_{L^q(\T,L^p(\M))}  \norm{\sum_{1\le n\le  N} e_n T^n x}_{L^p(\T,L^p(\M))}\\
			&\le e^4E_p^{(2)}N^{\delta_p+\delta_q} \log^{\kappa_p+\kappa_q}(N+1) \|x\|_{L^p(\M)}\|x^*\|_{L^q(\M)}
			\\
			&  = e^4D_{SK}^2 E_p^{(2)} N^{\left|\frac{1}{2}-\frac{1}{p}\right|} \log^{\kappa}(N+1)\|x\|_{L^p(\M)}\|x^*\|_{L^q(\M)}, 
		\end{align*}
		with
		$$
		\kappa := \kappa_p+\kappa_q.
		$$
		The result follows by taking the supremum over $x^*$ and $x$ with 
		$\|x\|_{L^p(\M)}= \|x^*\|_{L^q(\M)}=1$. 
	\end{proof}
	\medskip

	\section{Optimality of the polynomial rate in Theorem \ref{main-theorem} }
	
	In this section we give an example showing optimality in Theorem \ref{main-theorem} and we prove Proposition \ref{example}.

	\medskip

	This example has been presented by Lubich and Nevanlinna \cite{LN} to prove that the bound $\|T^N\|=\underset{N\to \infty}\OO(\sqrt N)$ is the best possible for strongly Kreiss bounded operators $T$ on general Banach spaces. More precisely, their example corresponds to the case where $p=\infty$ in the construction below.
	
	\medskip
	
	Lubich and Nevanlinna provided only an outline of the proof partly based on a book by Brenner, Thom\'ee and Wahlbin \cite{BTW}, which deals with Fourier analysis on the real line while the example  is defined on $\Z$. Hence, we decided to give here a detailed proof independent from the book \cite{BTW}, at least from its results.
	
	\medskip

	We denote by $A(\TT)$ the space of  continuous functions $f$ on $\TT$  having an 
	absolutely convergent Fourier series. 
	
	\medskip
	
	
	For  $f \in A(\TT)$, define 
	$f(R)$ where $R$ is the right shift on any of $\ell^p(\Z)$, $1\le p \le 
	\infty$, by the Hille-Phillips calculus:
	$$
	f(R) := \sum_{n\in \Z} c_n(f)R^n,
	$$
	where ${\mathbf{c_f}} := (c_n(f))_{n\in \Z}$ is the sequence of Fourier coefficients of $f$. We use here the convention that 
	$c_n(f)= \frac1{2\pi}\int_\TT f(t)e_{-n}(t)dt$. 
	\medskip

	Notice that 
	$$
	f(R) = T_{\bf c_f},  
	$$
	where for $\mathbf{m}:= (m_n)_{n\in \Z} \in \ell^{1}(\Z)$, the operator $T_{\mathbf{m}}$ is the convolution operator defined by
	
	$$\begin{array}{ccccc}
		T_\mathbf{m} & : & \ell^\infty(\Z) & \to & \ell^\infty(\Z) \\
		& & \mathbf{b} & \mapsto & \mathbf{m} \star \mathbf{b}.\\
	\end{array}
	$$
	It is known that 
	\begin{equation}\label{Tminfty}
		\norm{T_\mathbf{m}}_{B(\ell^{\infty}(\Z))} = \sum_{n\in \Z} |m_n|
	\end{equation}
	and that for every $1\le p <\infty$, 
	\begin{equation}\label{Tminfty}
		\norm{T_\mathbf{m}}_{B(\ell^{p}(\Z))} \le \norm{T_\mathbf{m}}_{B(\ell^{\infty}(\Z))}.
	\end{equation}
	Moreover, for $1\le p \le \infty$ and $q= \frac{p}{p-1}$, 
	\begin{equation}\label{Tmdual}
		\norm{T_\mathbf{m}}_{B(\ell^{p}(\Z))} = \norm{T_\mathbf{m}}_{B(\ell^{q}(\Z))}.
	\end{equation}
	
	\medskip
	
	Now, let us describe the function $f$ that will be used to construct our operator $V$.

	\medskip 
	
	Let  $a$ be a complex number with  $0<|a|<1$ and let $\varphi\in \R$. Set  
	$$
	q_{a,\varphi}(z):={\rm e}^{i\varphi}\frac{z-a}{1-\bar{a}z}, 
	\quad z \in \overline{\DD}$$
	
	The function $q_{a,\varphi}$ is a so-called M\"obius transformation mapping bijectively $\overline{\DD}$ onto itself, that is analytic from $\DD$ to $\DD$ and continuous from  $\overline{\DD}$ to $\overline{\DD}$.
	
	\medskip
	
	We associate with it a function from $A(\TT)$ by setting
	
	$$
\psi_{a,\varphi}:=q_{a,\varphi}\circ e_1\, .	
	$$

	In the sequel, when no confusion is possible, we simply write $\psi$ and $q$ instead of, respectively, $\psi_{a,\varphi}$ and $q_{a,\varphi} $.
	
	\medskip
	
	We define the operator  $V$, by setting
	$$
	V:= \psi(R).
	$$

	We shall now prove that $V$ enjoys all the properties described 
	in Proposition \ref{example}. The proof will be separated into three sections. 
	
	\medskip
	
	\subsection{Proof of the lower bound in \eqref{upper-lower}.}
	~~
	
	
	\medskip
	
	As mentioned in \cite{LN} one can establish the lower
	bound in \eqref{upper-lower}  for the operator $V$ by recasting the problem as that of obtaining a
	lower bound for the iterates of an operator on  $L^p(\R)$, using results of \cite{BTW}. Here,
	we present a more direct proof that does not rely on \cite{BTW}, except 
	for the use of van der Corput's lemma. Nonetheless, the core arguments are similar to those in \cite{BTW}. 
	
	\medskip

	In view of \eqref{Tmdual} (notice that $\tau_q=\tau_p$ if $1/p+1/q=1$) it is enough to prove the result for  $p\in [1,2]$. We consider this case now.

	\medskip

	Let $f \in A(\TT) $, to be chosen later. By Parseval's equality 
	and H\"older's inequality with $q=p/(p-1)$, for every $N\in \N$,
	\begin{align}
		  \label{first-BTW} \norm{f}_{L^2(\TT)}^2 &=  \norm{\psi^Nf}_{L^2(\TT)}^2 \\
		\nonumber     &= 2\pi\sum_{n \in \Z} c_n(\psi^Nf)\overline{c_n(\psi^Nf)} \\
		\nonumber     &\le 2\pi\norm{\big(c_n(\psi^Nf)\big)_{n\in \Z}}_{\ell^p(\Z)}\norm{\big(c_n(\psi^Nf)\big)_{n\in \Z}}_{\ell^{q}(\Z)} \\
		\nonumber   & \le 2\pi \|V^N\|_{B(\ell^p(\Z))}\norm{(c_n(f))_{n\in \Z}}_{\ell^{p}(\Z)}\norm{(c_n(\psi^Nf))_{n\in \Z}}_{\ell^{q}(\Z)}.
	\end{align}
	By Hölder's inequality, for every $N\in \N$ we have
	\begin{align}
		 \label{second-BTW}\norm{\big(c_n(\psi^Nf)\big)_{n\in \Z}}_{\ell^{q}(\Z)} &\le \norm{\big(c_n(\psi^Nf)\big)_{n\in \Z}}_{\ell^{2}(\Z)}^{\frac{2}{q}}\norm{\big(c_n(\psi^Nf)\big)_{n\in \Z}}^{1-\frac{2}{q}}_{\ell^{\infty}(\Z)}  \\
		\nonumber  &= \frac{\norm{f}_{L^2(\TT)}^{\frac{2}{q}}}{(2\pi)^{\frac{1}{q}}}\norm{\big(c_n(\psi^Nf)\big)_{n\in \Z}}^{\frac{2}{p}-1}_{\ell^{\infty}(\Z)}.
	\end{align}
	Combining \eqref{first-BTW} and \eqref{second-BTW} and using that $\norm{(c_n)_{n\in \Z}}_{\ell^{p}(\Z)} \le \norm{(c_n)_{n\in \Z}}_{\ell^{1}(\Z)}$, we obtain, for every $N\in \N$,
	$$
	\|V^N\|_{B(\ell^p(\Z))}\ge \frac{(2\pi)^{1-\frac{1}{q}}\norm{f}_{L^2(\TT)}^{1-\frac{2}{q}}}{\norm{(c_n(f))_{n\in \Z}}_{\ell^{1}(\Z)}}\norm{\big(c_n(\psi^Nf)\big)_{n\in \Z}}^{1-\frac{2}{p}}_{\ell^{\infty}(\Z)} .
	$$

	Let us assume for a moment that there exists $D=D_f$ such that for every $N \in \N$,
	\begin{equation}\label{findf}
		\norm{\big(c_n(\psi^Nf)\big)_{n\in \Z}}_{\ell^{\infty}(\Z)} \le \frac1{DN^{1/2}}\, .
	\end{equation}
	Then, for every $N \in \N$,
	$$
	\|V^N\|_{B(\ell^p(\Z))}\ge \frac{\min (1, 2\pi \|f\|_{L^2(\TT)})}{\norm{(c_n(f))_{n\in \Z}}_{\ell^{1}(\Z)}}\min(1, D)N^{1/p-1/2}\, ,
	$$
	which is exactly the desired lower bound in \eqref{upper-lower}, provided that $f$ is not the null function. 
	
	\medskip
	
	It remains to prove that we can find a function $f\in A(\TT)$, $f\not\equiv 0$, that satisfies \eqref{findf}.
	
	
	\medskip
	
	For every $t\in \TT$, set $\sigma(t):= -i \int_0^t \frac{\psi'}\psi(u)du$. Then, $\sigma$ is twice continuously differentiable on $\TT$ and 
	$\psi(t)={\rm e}^{i\sigma(t)}$, for every $t\in \TT$.  Moreover, since $|\psi|\equiv 1$, $\sigma$ is real-valued.

	\medskip

	Since $q$ is not a rotation $\sigma''$ is not identically equal to 0 and there exist $-\pi < \alpha< \beta < \pi$ and $c>0$ such that
	$|\sigma''(t)| \ge c$, for every $t\in [\alpha,\beta]$.

	Let  $f$ be any function in $ A(\TT)$, not identically equal to $0$, such that $f(t) = 0$ if $t\in \T \backslash [\alpha,\, \beta]$ and $f \in C^1((\alpha,\beta))$. 
	
	\medskip

	Integrating by parts, we obtain
	\begin{align*}
		2\pi c_n(\psi^Nf)   &  = \int_{-\pi}^{\pi} f(t)\psi^N(t)e^{-int}dt = \int_{\alpha}^{\beta} f(t)\psi^N(t)e^{-int}dt \\
        & =\int_{\alpha}^{\beta} f'(t)\int_{\alpha}^t \psi^N(x)e^{-inx}dx dt\, \\
		& = \int_{\alpha}^{\beta} f'(t)\int_{\alpha}^t e^{-inx+iN\sigma(x)}dx dt\, 
	\end{align*}
	for all $n$ and $N$ from $\N$. By van der Corput's lemma (see \cite[Lemma 1.5.1]{BTW}), there exists $c>0$ such that for all $n,N\in \N$ and $t\in [\alpha, \beta]$, 
	$$
	\Big|\int_{\alpha}^t e^{-inx+iN\sigma(x)}dx \Big|\le 8c^{-1/2}N^{-1/2},
	$$
	so that 
	$$
	\sup_{n\in \Z} |c_n(\psi^Nf)| \le C\norm{f'}_{L^1(\TT)}N^{-1/2}
	$$
	with $C = \frac{4c^{-1/2}}{\pi}$, which is exactly \eqref{findf}.

	\medskip
	
	\subsection{Proof that $V$ is strongly Kreiss bounded on $\ell^{p}(\Z)$}

	~~
	
	\medskip

	By \eqref{Tminfty} it suffices to prove that $V$ is strongly Kreiss bounded on $\ell^{\infty}(\Z)$.  The proof of that fact  has been outlined by 
	Lubich and Nevanlinna \cite{LN}. We propose below a complete and detailed proof, from which the uniformity of the obtained bounds with respect to the different parameters involved will be clear.

	\medskip
	
	For $\lambda \in \C$ with $|\lambda|> 1$ and $k\in \N$,  consider the function $\chi_{\lambda,k}$  defined by
	$$
	\chi_{\lambda,k} := \Bigg(\frac{|\lambda|-1}{\lambda-\psi}\Bigg)^k\, .
	$$
	Since 
	$$
	(|\lambda|-1)^k{R(\lambda,V)^k} = \chi_{\lambda,k}(V),
	$$
	to prove that $V$ is strongly Kreiss bounded on $\ell^{\infty}(\Z)$, it suffices to control the $\ell^1$-norm of the sequence of Fourier coefficients $(c_n(\chi_{\lambda,k}))_{n\in \Z}$, uniformly with respect to $\lambda$ and $k$.  
	
	\medskip
	
	Write $\lambda=|\lambda|{\rm e}^{i\alpha}$ for some $\alpha\in \R$ and let $\tau\in \R$ be such that $\psi(\tau )={\rm e}^{i\alpha}$.

	\medskip
	
	Clearly, it is enough to control the $\ell^1$-norm of the Fourier coefficients of 
	$$ \TT \ni t\mapsto\Bigg(\frac{|\lambda|-1}{|\lambda|- 
		{\rm e}^{-i\alpha}\psi(t+\tau)}\Bigg)^k,
	$$ uniformly with respect to $k\in \N$ and $\lambda\in \C$, $|\lambda| > 1$.
	
	\medskip

	Notice that 
	$${\rm e}^{-i\alpha}\psi_{a,\varphi}(\cdot +\tau)= \psi_{a{\rm e}^{-i\tau}, \varphi -\alpha}.
	$$
	
	\medskip
	
	In particular, it is sufficient to assume that $\lambda$ is a real number satisfying $\lambda>1$ and that $\psi(0)=1$. Moreover, since $a$ and $a {\rm e}^{-i\tau}$ have the same modulus, for our purpose it is enough to obtain a bound depending only on $|a|$, and uniform with respect to $\varphi\in \R$.
	
	\medskip
	In view of our assumption, in particular, we have
	\begin{equation}\label{calc}
		\frac{{\rm e}^{i\varphi}}{1-\bar a}
		=\frac{1}{1-a}\, .
	\end{equation}
	
	For brevity, we will write $\chi_k$ instead of $\chi_{\lambda,k}$ throughout the sequel. All variables indexed by $k$ will be understood to depend implicitly on both $k$ and $\lambda$.

	\medskip

	The proof makes use of the following three lemmas, which we state now and will prove in the appendix. The first one is a standard result in the spirit of Carlson's inequality.
	\begin{lemma}\label{carlson-lemma}
		Let $h \in A(\TT)$ be such that $h' \in L^2(\TT)$, where $h'$ stands for the derivative of $h$ in the sense of distributions.
		For every real number $x\ge 1$, we have 
		\begin{equation}\label{carlson-bis}
			\sum_{n\in \Z}|c_n(h)|\le \sqrt 2(x+1)\|h\|_{L^2(\TT)}+ \frac{\sqrt 2}{x}
			\|h'\|_{L^2(\TT)}\, .
		\end{equation}
	\end{lemma}

	The next lemma was stated in \cite[Example 2.2]{LN}.

\begin{lemma}\label{auxilliary}
		Let $\psi_{a,\varphi}$ be as above. There exists $\gamma >0$, depending only on $|a|$, such that for all $t\in \TT$ and $\lambda>1$, 
		\begin{equation}\label{aux-bound}
			\Bigg|\frac{\lambda -1}{\lambda-\psi_{a,\varphi}(t)}\Bigg|\le \frac1{\big(1+\frac{\gamma t^2}{\mu}\big)^{1/2}}\, , 
		\end{equation}
		where $\mu=\min (\lambda-1, (\lambda-1)^2)$.
	\end{lemma}
	

	\begin{lemma}\label{theta}
		For all $k\in \N$, $\lambda>1$, $\varphi \in \R$ and $a\in \DD\backslash\{0\}$, there exist a $C^1$ function $\theta_k=\theta_{k,\lambda}$ on $\TT$ and a constant $C>0$ depending only on $|a|$ such that for all
		$t\in \TT$,
		\begin{gather}
			\label{theta-infty} |\theta_{k}(t)|\le 1\,  ;\\
			\label{estthetachik}
			\big|\big(\theta_k\chi_k\big)'(t)\big| \le \frac{Ck|t|}{\mu\big(1+\frac{\gamma t^2}{\mu}\big)^{k/2}}\, ;\\
			\label{theta-ell1}
			\sum_{n\in \Z}\Big|c_n\Big(\frac{1}{\theta_k}\Big)\Big| \le 2\, ,  
		\end{gather}
		
	\end{lemma}
	
	\begin{remark}
	Here, by $C^1$, we mean that $\theta$ may be extended to a $2\pi$-periodic $C^1$ function on the whole real line.
	\end{remark}

	\medskip

	We are now ready to prove that $V$ is strongly Kreiss bounded on 
	$\ell^\infty(\Z)$.
	
	\medskip
	
	We proceed to show that there exists $C>0$ depending only on $|a|$,  such that for every  $k\in \N$ and every real number $\lambda > 1$,
	\begin{equation}\label{cnchibounded}
		\sum_{n\in \Z} |c_n(\chi_{k})| \le C.
	\end{equation}

	First of all, using \eqref{theta-ell1},
	\begin{align}
		\label{controlchik}\sum_{n \in \Z} |c_n( \chi_k)| &= \sum_{n \in \Z} \Big|c_n\Big( \frac{1}{\theta_k}\theta_k\chi_k\Big)\Big| =\sum_{n \in \Z} \Big| \sum_{r\in \Z} c_r\Big(\frac{1}{\theta_k}\Big)  c_{n-r}(\theta_k\chi_k)\Big|  \\
		\nonumber&\le \Big(\sum_{n \in \Z} \Big|c_n\Big(\frac{1}{\theta_k}\Big)\Big| \Big) \Big( \sum_{n \in \Z}  |c_n(\theta_k\chi_k)|\Big) \le 2\sum_{n \in \Z}  |c_n(\theta_k\chi_k)|.
	\end{align}
	Moreover, by \eqref{aux-bound} and \eqref{theta-infty}, we see that (using the change of variable $u=\sqrt{k/\mu} t$), 
	\begin{align}
		\label{first-control}\int_{-\pi}^\pi |\theta_k\chi_k|^2(t)dt &\le \int_{-\pi}^\pi\frac{dt}{(1+\frac{\gamma t^2}{\mu})^{k}} \\
		\nonumber&\le \Big(\frac{\mu}k\Big)^{1/2}
		\int_{-\infty}^{+\infty} \frac{du}{(1+\frac{\gamma u^2}k)^k} \\
		\nonumber&\le \Big(\frac{\mu}k\Big)^{1/2}
		\int_{-\infty}^{+\infty} \frac{du}{(1+\gamma u^2)}
		\\
		& \nonumber\le C \Big(\frac{\mu}k\Big)^{1/2}\, , 
	\end{align}
	where we used that for every $x>0$, $y\mapsto (1+x/y)^y$ is increasing on $[1,+\infty)$.
	
	\medskip
	
	On the other hand, $|\theta_k \chi_k|\le 1$, so that 
	\begin{equation}\label{first-control-bis}
		\int_{-\pi}^\pi |\theta_k\chi_k|^2(t)dt\le \min \Big(2\pi,  C \Big(\frac{\mu}k\Big)^{1/2}\Big)\, .
	\end{equation}

	We shall now estimate the $L^2$-norm of the derivative of 
	$\theta_k \chi_k$. Proceeding as above, using \eqref{estthetachik},  we obtain  

	\begin{align}
		\label{first-derivative}\int_{-\pi}^\pi |(\theta_k\chi_k)'|^2(t)dt &   \le \frac{4D^2k^2}{\mu^2} \Big(\frac{\mu}{k+1}\Big)^{3/2}\int_{-\infty}^{+\infty}\frac{u^2 du}{(1+\frac{\gamma u^2}{k+1})^{k+1}} \le \frac{C \sqrt k }{\sqrt \mu}\, .
	\end{align}

	If $C \Big(\frac\mu{k}\Big)^{1/2}<2\pi$, we apply \eqref{carlson-bis} 
	with $x^2 = \frac{k}{C^2\mu}$ and obtain
	\begin{align}
		\label{first-case} \sum_{n\in \Z} |c_n(\theta_k\chi_k)| &\le \sqrt{2}\Bigg(\sqrt{\frac{k}{C^2\mu}}+1\Bigg)\times C\Big(\frac\mu{k}\Big)^{1/2} + \sqrt{\frac{C^2\mu}{k}}\times \frac{C\sqrt{k}}{\sqrt{\mu}}
		\\\nonumber &\le \sqrt{2}\Big(1+C\Big(\frac\mu{k}\Big)^{1/2}\Big) + C^2 \le \sqrt{2}(1+2\pi) + C^2\, .
	\end{align}
	
	\medskip

	
	If $C \Big(\frac\mu{k}\Big)^{1/2}\ge 2\pi$, we apply \eqref{carlson-bis} 
	with $x=1$ and obtain
	\begin{align}\label{second-case}
		\sum_{n\in \Z} |c_n(\theta_k\chi_k)| &\le 4\sqrt{2}\pi + \frac{C\sqrt{k}}{\sqrt{\mu}}
		\le4\sqrt{2}\pi + \frac{C^2}{2\pi}.
	\end{align}
	
	Combining \eqref{first-case} and \eqref{second-case}, we infer that there exists $C>0$ such that  
	$$
	\sum_{n\in \Z} |c_n(\theta_k\chi_k)| \le C
	$$
	which, by \eqref{controlchik}, implies \eqref{cnchibounded}. Hence $V$ is strongly Kreiss bounded on every $\ell^p(\Z)$$, 1\le p\le \infty$.
	
	
	\subsection{Proof of the upper bound in \eqref{upper-lower}}
	
	Since $V$ is strongly Kreiss bounded on $\ell^{\infty}(\Z)$, by a result of Lubich and Nevanlinna \cite{LN} recalled in the introduction, there exists $C>0$ such that for every $N \in \N$
	$$
	\norm{V^N}_{B(\ell^{\infty}(\Z))} \le CN^{1/2}\, .
	$$
	Now, $V$ is an isometry on $\ell^{2}(\Z)$, hence, by the  Riesz-Thorin interpolation theorem applied to $V^N$ 
	(for every $N \in \N$), we infer that the upper bound in \eqref{upper-lower} holds for every $p\in [2,\infty]$. The case where $p\in [1,2)$ then follows from \eqref{Tmdual}.

	\section{Some particular strongly Kreiss bounded operators}
	
	We consider here the case of \emph{positive} strongly Kreiss bounded operators on $L^p(\M)$ (we no longer require $\mu$ to be $\sigma$-finite)  or \emph{absolutely} strongly Kreiss bounded operators (on any Banach space). The proof makes use of an idea on obtaining norm bounds on powers of Kreiss operators from  \cite{AC}. 
	\medskip
	
	Let us begin with a general result.
	\begin{lemma}\label{general}
		Let $1\le p \le 2$, $C,D \ge 1$ and $\alpha > 0$. Let $T$ be a bounded operator on a Banach space  $X$ such that for every $N\in \N$ and every $x\in X$
		\begin{equation}\label{general-assumption2}
			\sum_{N+2 -2\sqrt{N} \le n \le N} \norm{T^nx}^p \le C^pN^{p/2}
			\norm{x},
		\end{equation} 
		and
		\begin{equation}\label{general-assumption1}
			\norm{T^N} \le DN^{\alpha}.
		\end{equation}
		Then, there exist $E>0$ and $\kappa >0$ such that for every $N\in \N$,
		\begin{equation}\label{resposLem}
			\|T^N\| \le EN^{1/q}\log^\kappa (N+1).
		\end{equation}
		
	\end{lemma}
	\begin{proof}
		We start with the following observation. For all $x\in X$ and
		$x^*\in X^*$,
		\begin{align*}
			(\lfloor2\sqrt N\rfloor-1)|\langle x^*, T^{N+1}x\rangle |^p  &  \le \sum_{1\le n\le 2\sqrt N-1}| \langle T^{*n}x^*, T^{N+1-n}x\rangle \big|_{X^*,X}^p\\ \nonumber
			&  \le \|x^*\|^p\max_{1\le n\le 2\sqrt N-1}\|T^{*n}\|^p
			\sum_{ N -2 \sqrt{N}+2\le n\le N}\|T^{n}x\|^p\, .
		\end{align*}
		Taking the supremum over $x,x^*$ of norm 1, using that $\|T^{*n}\|=\|T^n\|$
		for every $N\in \N$,
		\begin{align*}\label{first-estimate}
			\|T^{N}\| & \le  (\lfloor2\sqrt{N}\rfloor-1)^{-1/p}\max_{1\le n\le 2\sqrt N-1} \|T^n\|\sup_{\|x\| \le 1} \Big(\sum_{ N -2 \sqrt{N}+2\le n\le N}\|T^{n}x\|^p\Big)^{1/p}. \\
			&\le N^{-1/(2p)}\max_{1\le n\le 2\sqrt N} \|T^n\|\sup_{\|x\| \le 1} \Big(\sum_{ N -2 \sqrt{N}+2\le n\le N}\|T^{n}x\|^p\Big)^{1/p}. 
		\end{align*}
		Using assumption \eqref{general-assumption2}, for all $N\in \N$,
		\begin{equation}\label{iteration}
			\|T^{N}\| \le CN^{1/(2q)}\max_{1\le n\le 2\sqrt N}\|T^n\|.
		\end{equation}
		By induction, we then obtain that for all $N,K\in \N$, 
		$$
		\|T^N\|\le D (2^{\alpha} C)^K N^{2^{-K}\alpha+(1-2^{-K})/q}\, ,
		$$
		Indeed, combining estimate \eqref{iteration} with \eqref{general-assumption1}, we obtain that for all $N \in \N$, 
		\begin{equation}\label{iterationbis}
			\|T^N\|\le DC2^{\alpha} N^{\alpha/2 +1/2q}\,,
		\end{equation}
		which gives the case $K=1$. For the inductive step from $K-1$ to $K$, we combine the inductive hypothesis with \eqref{iteration}.
		Finally, we conclude as in the proof of Theorem \ref{main-theorem} to obtain \eqref{resposLem}.
	\end{proof}
	From Lemma \ref{general} we deduce a better bound  than the one obtained in Theorem \ref{main-theorem} for  strongly Kreiss bounded positive operators on $L^p(\M)$ when $p\in [1,4/3)\cup (4,+\infty)$ (the cases where $p = 1$ and $p=\infty$ are discussed below).
	\begin{proposition}\label{positive}
		Let $T$ be a positive operator that is strongly Kreiss bounded on $L^p(\M)$, $1\le p<\infty$. Then there exist $C,\kappa>0$ such that for every $N\in \N$,
		$$\|T^N\|\le CN^{1/\bar p}\log^\kappa (N+1),$$ 
		where $\bar p=\max(p,p/(p-1))$.
	\end{proposition}
	
	The proof is straightforward using the following lemma and the fact that every strongly Kreiss bounded operator satisfies \eqref{general-assumption1} for $\alpha = 1/2$.
	
	\begin{lemma}
		Let $1\le p <\infty$. Any positive strongly Kreiss bounded operator $T$ on $L^p(\M)$ satisfies \eqref{general-assumption2} for every $x \in L^p(\M)$. 
	\end{lemma} 
	\begin{proof}
		Using  the fact that $\|\cdot \|_{\ell^p}\le \|\cdot \|_{\ell^1}$, the positivity of $T$, the Krivine calculus (\cite[Theorem 1.d.1]{LT79} and Lemma \ref{technical},  we obtain that for all $N\in \N$ and all nonnegative $x\in L^p(\M)$,
		\begin{align*}
			\Big( \sum_{n\ge 0}\frac{N^n T^n x}{n!}\Big)^p &\ge \sum_{n\ge 0} \Big(\frac{N^n T^n x}{n!}\Big)^p  \\ \nonumber 
            &\ge \sum_{ N -2 \sqrt{N}+2\le n\le N}\Big(\frac{N^n T^n x}{n!}\Big)^p\, .
            \\ & \ge \frac{C_{exp}^p{\rm e}^{pN}}{ N^{p/2}} \sum_{ N -2 \sqrt{N}+2\le n\le N} (T^n x)^p. \nonumber          
		\end{align*}\label{arnold-coine}
		Integrating with respect to $\mu$ and using the strong Kreiss boundedness of $T$, it follows that for all $N\in \N$ and all nonnegative $x \in L^p(\M)$,  
		\begin{align*}
			C_{SK}e^{pN}\norm{x}^p &\ge \norm{e^{NT}x}^p = \int_{\M} \Big( \sum_{n\ge 0}\frac{N^n T^n x}{n!}\Big)^p d\mu \\
			&\ge \frac{C_{exp}^p{\rm e}^{pN}}{ N^{p/2}}\int_{\M} \sum_{ N -2 \sqrt{N}+2\le n\le N} (T^n x)^p d\mu \\
            &= \frac{C_{exp}^p{\rm e}^{pN}}{N^{p/2}} \sum_{ N -2 \sqrt{N}+2\le n\le N} \norm{T^nx}^p.
		\end{align*}
		Recall that $C_{SK}$ denotes the smallest constant $L>0$ such that \eqref{strong-kreiss} holds.
		This implies that \eqref{general-assumption2} is true for all $x\in L^p(\M)$ with $C = \frac{C_{SK}}{C_{exp}}$.  
		
	\end{proof}

	We turn now to the special case of absolutely strongly Kreiss bounded operators. Let $T$ be a bounded operator on $X$. We say that $T$ is absolutely strongly Kreiss bounded if there exists $C>0$ such that 
	$$
	\sum_{n=0}^{\infty}\frac{r^n}{n!}\norm{T^nx} \le Ce^r\norm{x}, \quad r>0, \, x\in X. 
	$$   
	These operators were first introduced in \cite[remarks p.17]{CCEL}. Such operators are clearly strongly Kreiss bounded (see \cite[Proposition 4.10]{CCEL}).  The operator $T$ defined in \eqref{exampleStrKreiss} provides an example of an operator which is absolutely strongly Kreiss and not power-bounded. Moreover, $T^*$ is strongly Kreiss bounded but cannot be absolutely strongly Kreiss, since in this case $T$ would be power-bounded (see \cite[Proposition 2.1]{CCEL}). Using Lemma \ref{technical}, for every $x \in X$, the operator $T$ satisfies  \eqref{general-assumption2} with $p = 1$. We can then apply Lemma \ref{general} to show that $\norm{T^N}$ has a logarithmic bound.
	\begin{proposition}\label{absstrong}
		Let $T$ be an absolutely strongly Kreiss bounded operator on $X$. Then, there exist $C,\kappa >0$ depending only on $C_{SK}$, such that for every $N\in \N$,
		\begin{equation}\label{absstrongbound}
			\norm{T^N} \le  C\log^\kappa(N+1).
		\end{equation}
		\begin{remark}
			When $X =L^p(\M)$, $1\le p <\infty$, the bound \eqref{absstrongbound} is sharp. Indeed, according to \cite[remark 1, page 16]{CCEL}, the operators $T_{\kappa}$, $\kappa >0$, defined by \eqref{exampleStrKreiss}, are absolutely strongly Kreiss bounded on $\ell^p(\Z)$ and satisfy 
			$$\norm{T_{\kappa}^N} = \frac{\log^\kappa(N+1)}{\log^\kappa(2)}$$
			for all $N\in \N$ and $\kappa >0$.
		\end{remark}
	\end{proposition}
	
	\medskip 
	
	We now discuss the case of positive strong Kreiss operator on (AL) and (AM)-spaces. We refer to \cite{Meyer-N}, Section 2, for more details. A Banach lattice $X$ is an $(\text{AL})$-space if the norm is additive on the positive cone on $X$, that is
	\begin{equation}\label{ALnorm}
		\quad \| x+y \| = \|x\| + \|y\|, \quad x,y \in X_+.
	\end{equation}
 A Banach lattice $X$ is an (AM)-space if the norm on $X$ satisfies
	$$
	\quad \| \sup(x,y) \| = \sup(\|x\|,\|y\|), \quad x,y \in X_+.
	$$
	If $X$ is an (AM)-space, then $X^*$ is an (AL)-space. 
	It is known that an (AL)-space is isometrically isomorphic to some $L^1$ space and that an (AM)-space is isometrically isomorphic  to some $C(K)$ where $K$ is a compact space. We are now ready to prove the following statement.
	
	\begin{proposition}\label{ALAMest}
		Let $T$ be a positive strongly Kreiss bounded operator on an (AL)-space or an (AM)-space. Then there exist $C,\kappa >0$ such that for every $N\in \N$,
		\begin{equation}
			\norm{T^N} \le C\log^\kappa (N+1).
		\end{equation}
	\end{proposition}
\begin{remark}
Since $L^\infty(\mathcal{M})$ is an (AM)-space, Proposition~\ref{positive} remains valid also for the case $p=\infty$ by the above result.
  \end{remark}
		\begin{proof}
			If $T$ is a positive strongly Kreiss bounded operator on an (AL)-space, then using \eqref{ALnorm}, it is straightforward that $T$ is absolutely strongly Kreiss bounded, and we can  then conclude invoking Lemma \ref{general}. 
			
			If $T$ is a positive strongly Kreiss bounded operator on an (AM)-space, then $T^*$ is a positive strongly Kreiss bounded on an (AL)-space. Then this case is reduced to the first part of the proof.
		\end{proof}
		
		\begin{question}
			In view of the considerations in Section
			5, it is natural to ask whether the bound in Proposition \ref{positive} can be improved to obtain a logarithmic bound. More precisely, for a positive strongly Kreiss bounded operator $T$ on $L^p(\M)$ with $ 1<p<\infty$, are there $C,\kappa>0$ such that for every $N\in \N$,
			$$\norm{T^N} \le C\log^{\kappa}(N+1) \;?$$ 
		\end{question}

		\section{Appendix}

		\subsection{Proof of Lemma \ref{carlson-lemma}} Let $x\ge 1$ be a real number. Set 
		$m:= \lfloor x^2\rfloor$. By the Cauchy-Schwarz inequality, we have 
		$$
		\sum_{|n|\le m+1} |c_n(h)|\le 
		\sqrt{2m+3}\Big(\sum_{|n|\le m+1}|c_n(h)|^2\Big)^{1/2} \le \sqrt{2m+3} 
		\|h\|_{L^2(\TT)}\,
		$$
		and
		\begin{align*}
			\sum_{|n|\ge m+2} |c_n(h)|&\le\Big( \sum_{|n| \ge m+2} 
			\frac1{|n|(|n|-1)}\Big)^{1/2} \Big(\sum_{|n|\ge m+2} n^2|c_n(h)|^2\Big)^{1/2} \\
			&\le \sqrt{\frac{2}{m+1}}\|h'\|_{L^2(\TT)}\, ,
		\end{align*}
		and the result follows. \qed

		\subsection{Proof of Lemma \ref{auxilliary}}
		
		We first prove that there exists $C>0$ such that 
		\begin{equation}\label{intermediary}
			\Bigg|\frac{\lambda -1}{\lambda-q(z)}\Bigg|\le \frac1{\Big(1+\frac{C |1-z|^2}{\mu}\Big)^{1/2}}\, , \quad |z| = 1 .
		\end{equation}
		The result will follow since, writing $z={\rm e}^{it}$ with 
		$t\in \T$, $\psi(t)=q({\rm e}^{it})$ and $|1-z|=2|\sin(t/2)|
		\ge 2|t|/\pi$.

		\medskip
		
		Notice that \eqref{intermediary} is equivalent to 
		$$
		\frac\mu{|1-z|^2}\Bigg(\Big|\frac{\lambda-q(z)}{\lambda -1}\Big|^2-1
		\Bigg)\ge C\, , \quad |z| = 1 \, . 
		$$
		Now, for every $z\in \C$ with $|z| = 1$,
		\begin{align*}
			\Bigg|\frac{\lambda-q(z)}{\lambda -1}\Bigg|^2-1  & = \frac{(\lambda-1)^2 + |1-q(z)|^2 +2(\lambda-1)\big(1-{\rm Re}(q(z))\big)}{(\lambda-1)^2} -1\\
			& =  \frac{2\big(1-{\rm Re}(q(z))\big)}{\lambda-1} 
			+\frac{|1-q(z)|^2}{(\lambda-1)^2}
			\\
			&  \ge \frac1\mu \min \big(2(1-{\rm Re}(q(z)),
			|1-q(z)|^2 \big)\, 
			\\ & = \frac1\mu
			|1-q(z)|^2 \,,
		\end{align*}
		where the last equality comes from 
		$$
		|1-z|^2 = 2(1-{\rm Re}  (z)),\quad  |z| = 1.
		$$
		Hence, we just have to prove that
		$$
		\frac{ |1-q(z)|^2}{|1-z|^2}\ge C \, , \quad  |z| = 1.
		$$
		However, since $q$ is a bijection of the unit circle $\{z\in \C, \, |z| = 1 \}$, this is equivalent to proving 
		\begin{equation}\label{inter-bis}
			\frac{|1-z|^2}{|1-q^{-1}(z)|^2}\ge C \, , \quad |z| = 1.
		\end{equation}
		
		It is well-known that for all $z\in \overline{\DD}$, $q^{-1}(z)={\rm e}^{i\varphi}\frac{z+a}
		{1+\bar a z}$. Since $q^{-1}(1)=1= {\rm e}^{i\varphi}\frac{1+a}
		{1+\bar a }$, we obtain
		$$
		1-q^{-1}(z) = 1-\frac{1+\bar{a}}{1+a}\frac{z+a}{1+\bar a z}= \frac{(1-|a|^2)(1-z)}{(1+a)(1+\bar{a}z)}, \quad z\in \overline{\DD}.
		$$
		Using the estimate
		$$
		|1+\bar{a}z| \ge 1-|a|, \quad |z| = 1,
		$$
		we deduce
		$$
		|1-q^{-1}(z)|^2=\frac{(1-|a|^2)^2|1-z|^2}{|1+\bar a|^2|1+\bar a z|^2}\le \frac{(1+|a|)^2}{(1-|a|)^2} |1-z|^2\, , \quad |z| = 1 .
		$$
		
		Hence, \eqref{inter-bis} holds with $C = \frac{(1-|a|)^2}{(1+|a|)^2}$, and \eqref{aux-bound} holds with $\gamma=4C/\pi^2 =4 \frac{(1-|a|)^2}{(1+|a|)^2}/\pi^2 $. \qed

		\subsection{Proof of Lemma \ref{theta}}
		
		\medskip
		
		Notice that for every $t\in \TT$, 
		\begin{equation}\label{psi'}
			\psi'(t)= {\rm e}^{i\varphi}\frac{i(1-|a|^2){\rm e }^{it}}
			{(1-\bar a {\rm e}^{it})^2} \, .
		\end{equation}
		and 
		\begin{equation}\label{psi''}
			\psi''(t)=-{\rm e}^{i(t+\varphi)}(1-|a|^2)
			\frac{1+\bar a {\rm e}^{it}}{(1-\bar a{\rm e}^{it})^3}\, .
		\end{equation}
		Hence, using \eqref{calc}, 
		\begin{equation}\label{alpha}
			\psi'(0)=i\frac{1-|a|^2}{|1-a|^2}
		\end{equation}
		Let $\alpha := \frac{1-|a|^2}{|1-a|^2} $.
		Fix $k\in \N$ and $\lambda>1$. 
		Set $m_k=m_{k,\lambda}:= \lfloor k\alpha/(\lambda-1)\rfloor$, and set
		$$\theta_k=\theta_{k,\lambda}:=e_{-m_k}
		\left[1- \frac14\Big(\frac{k\alpha}{\lambda-1}-m_k\Big)  +\frac14\Big(\frac{k\alpha}{\lambda-1}-m_k\big)e_{-4}\Big) \right]\, .
		$$
		We shall see that this function satisfies  all the required properties. 
		
		\medskip
		
		Clearly, for every $t\in \T$, we have $|\theta_k(t)|\le 1=\theta_k(0)$ and therefore \eqref{theta-infty} holds. 
		
		\medskip
		
		Let us prove \eqref{estthetachik}.

		\medskip

		Setting $u_k:= e_{m_k}\theta_k$, we have 
		\begin{align}\label{theta'}
			\theta_k' &=  -im_k \theta_k-i e_{-m_k}\Big( \frac{k\alpha}{\lambda-1}-m_k\Big)e_{-4}\, \\
			&  = -ie_{-m_k} \Big(m_k u_k(t)+ \Big(\frac{k\alpha}{\lambda-1}-m_k\Big)e_{-4}\Big)\, .
		\end{align}
		Hence, 
		\begin{equation}\label{theta0}
			\theta_k'(0)=-ik\frac{\alpha}{\lambda-1}.
		\end{equation}
		Moreover
		\begin{align*}
			(\theta_k\chi_k)' & =\frac{(\lambda-1)^k}{(\lambda -\psi)^{k+1}}\big(\theta_k'(\lambda-\psi)+ k\theta_k \psi'\big)\\
			& 
			= \frac{(\lambda-1)^k}{(\lambda -\psi)^{k+1}}e_{-m_k}\Big(ku_k\psi' -i\Big( m_k u_k +\Big(\frac{k\alpha}{\lambda -1}-m_k\Big) e_{-4}\Big)(\lambda-\psi)\Big)
			\\   &  = \frac{(\lambda-1)^ke_{-m_k}}{(\lambda -\psi)^{k}}w_k\, ,
		\end{align*}
		where
		$$
		w_k := \frac{ku_k\psi' -i\Big( m_k u_k +\Big(\frac{k\alpha}{\lambda -1}-m_k\Big) e_{-4}\Big)(\lambda-\psi)}{\lambda-\psi}.
		$$
		
		Using \eqref{theta0} and the fact that $\psi'(0)=i\alpha$, we have
		\begin{align}\label{wk}
			w_k(0)&=k\theta_k(0)\psi'(0)-im_k -i\Big(\frac{k\alpha}{\lambda -1} -m_k\Big) (\lambda - \psi(0))= 0.
		\end{align}
		
		Moreover, setting $g_k := u_k\psi'$, we have
		\begin{align*}
			w_k' =&im_k\, u_k'+u_k''+k\frac{g'(\lambda-\psi) +g_k\psi'^2}{(\lambda-\psi)^2}\, .
		\end{align*}
		
		Notice that  $|u_k'|\le 1$ and $|u_k''|\le 4$.

		\medskip
		
		From \eqref{psi'} and \eqref{psi''}, it follows that for every $t \in \T$, 
		$$|\psi'(t)|\le \frac{1-|a|^2}{(1-|a|)^2} \quad \text{and}  \quad  |\psi''(t)|\le \frac{(1+|a|)^2}{(1-|a|)^2}.$$ 
		
		\medskip

		In particular, we infer that 
		\begin{align*}
			|w_k'|&\le C\Bigg(m_k+1+k\Big(\frac{1}{\lambda-1} + \frac{1}{(\lambda-1)^2}\Big)\Bigg) 
			\\& \le D\Big( m_k+\frac{k}{\mu} \Big) \le 2D\frac{k}{\mu},
		\end{align*} for some constants $C,\, D>0$ depending only on $a$ (and not on $k\in \N$ and $\lambda>1$).
		
		\medskip
		
		Combining this bound with \eqref{wk},  we see that for every $t\in \TT$,
		\begin{align}
			\label{estim}|w_k(t)| &  \le 2D \frac{k}{\mu}|t|.
		\end{align}
		Then \eqref{estthetachik} follows from \eqref{estim} and 
		\eqref{aux-bound}.

		\medskip
		
		It remains to prove \eqref{theta-ell1}. We set $r_k = \frac{k\alpha}{(\lambda-1)}-m_k$. For every  $k\in \N$ and $\lambda > 1$, we have
		$$
		\frac{1}{\theta_k(t)} = \frac{e_{m_k}}{1-\frac{r_k}{4}}\sum_{n \in \N}\Bigg(-\frac{r_k}{4(1-\frac{r_k}{4})}\Bigg)^ne_{-4n}.
		$$
		Therefore,
		\begin{align*}
			\sum_{n\in \Z}\Big|c_n\Big(\frac{1}{\theta_k}\Big)\Big|  &  \le \frac{1}{1-\frac{r_k}{4}} \sum_{n \in \N}\Bigg(\frac{r_k}{4(1-\frac{r_k}{4})}\Bigg)^n\\
			& \qquad  =\frac{ 4}{4-r_k} \times \frac{1}{1-\frac{r_k}{4-r_k}} = 
			\frac{2}{2-r_k}\le 2\, . 
		\end{align*}
		\qed

	\end{document}